\documentclass{amsart}
\usepackage{amssymb,a4wide}
\usepackage[all]{xy}
\usepackage{todonotes}

\newcommand{\IR}{\mathbb R}
\newcommand{\IN}{\mathbb N}

\newcommand{\IQ}{\mathbb Q}
\newcommand{\IC}{\mathbb C}

\newcommand{\A}{\mathcal A}

\newcommand{\M}{\mathcal M}

\newcommand{\w}{\omega}
\newcommand{\Hom}{\mathrm{Hom}}

\newcommand{\id}{\mathrm{id}}
\newcommand{\e}{\varepsilon}
\newcommand{\cB}{\mathcal B}

\newtheorem{theorem}{Theorem}[section]
\newtheorem{problem}[theorem]{Problem}
\newtheorem{lemma}[theorem]{Lemma}
\newtheorem{corollary}[theorem]{Corollary}
\newtheorem{proposition}[theorem]{Proposition}
\theoremstyle{definition}
\newtheorem{definition}[theorem]{Definition}
\newtheorem{example}[theorem]{Example}
\newtheorem{remark}[theorem]{Remark}

\newcommand{\K}{\mathcal K}

\title[Semitopological modules]{Semitopological modules}
\author{Taras Banakh and Alex Ravsky}
\address{T. Banakh: Ivan Franko National University of Lviv (Ukraine), and Jan Kochanowski University in Kielce (Poland)}
\email{t.o.banakh@gmail.com}
\address{A. Ravsky: Pidstryhach Institute for Applied Problems of Mechanics and Mathematics National Academy of Sciences of Ukraine}
\email{alexander.ravsky@uni-wuerzburg.de}

\dedicatory{Dedicated to the memory of V.K.~Maslyuchenko}

\begin{document}
\begin{abstract} Given a topological ring $R$, we study semitopological $R$-modules, construct their
completions, Bohr and borno modifications. For every topological space $X$, we construct the free
(semi)topological $R$-module over $X$ and prove that for a $k$-space $X$ its  free
semitopological $R$-module is a topological $R$-module. Also we construct a Tychonoff space $X$
whose free semitopological $R$-module is not a topological $R$-module.
\end{abstract}
\maketitle

\section{Introduction}

This paper was motivated by problems on semitopological linear spaces posed by Vo\-lo\-dy\-myr
Kyrylovych Mas\-ly\-uchenko in \cite{Masl}. A {\em semitopological linear space over a topological
field $F$} is a linear space $X$ over the field $F$, endowed with a Hausdorff topology $\tau$
turning $X$ into an Abelian topological group and making the multiplication $F\times X\to X$,
$(\lambda,x)\mapsto\lambda\cdot x$, separately continuous. If this operation is jointly
continuous, then $X$ is a {\em topological linear space} over $F$. By a {\em semitopological linear
space} we understand a semitopological linear space over the field $\IR$ of real numbers endowed with the Euclidean topology. These spaces (under the name ``$N$-spaces") were
introduced and studied by V.K.~Mas\-lyu\-chen\-ko in \cite{Masl}, where he posed the following

\begin{problem}\label{prob} Is each (metrizable) semitopological linear space a topological linear space?
\end{problem}

In \cite{BMR} we proved that each metrizable semitopological linear space is a topological linear
space and also constructed a semitopological linear space which is not a topological linear space.
These results motivate the authors to study semitopological linear spaces in more details and more generality, so the present paper addresses exactly this problem.

In Section~\ref{s3} we introduce semitopological $R$-modules over topological rings $R$, study
their completions (in Theorem~\ref{t3.1}), detect their bounded subsets (in Theorem~\ref{t3.2}),
and establish joint continuity properties of their multiplication maps (in Theorems~\ref{t3.2} and
\ref{t3.4}).

In Section~\ref{s4} we study Bohr modifications  of semitopological modules and construct many
examples of semitopological modules which are not topological modules.

In Section~\ref{s5} we discuss bornologies on topological rings and
construct so-called bornomo\-di\-fi\-ca\-ti\-ons of semitopological modules, turning them into
topological modules.

So, we have two constructions over semitopological modules, which have opposite effects: the
Bohr modifications of  semitopological linear spaces produce semitopological linear spaces which
are not topological linear spaces, whereas the bornomodifications transform semitopological linear
spaces into topological linear spaces.

The properties of the bornomodifications are used in the final Section~\ref{s6} devoted to free semitopological modules over topological spaces. The main result proved in this section is Theorem~\ref{t6.2} implying that the free semitopological linear space over a $k$-space $X$ coincides with the free topological linear space over $X$. On the other hand, the free semitopological linear space over the real line endowed with the Bohr topology is not a topological linear space.

\section{Preliminaries}\label{s2}

In this section we shall make some conventions and recall the necessary information on topological spaces and topological groups.

\subsection{Baire and strongly Baire spaces} {\em All topological spaces considered in this paper are Tychonoff}. For a subset $A$ of a topological space $X$ we denote by $\bar A$ the closure of $A$ in $X$.

A topological space $X$ is called {\em Baire} if for any sequence $(U_n)_{n\in\w}$ of open dense
subsets of $X$ the intersection $\bigcap_{n\in\w}U_n$ is dense in $X$. Baire spaces can be
characterized using the Choquet game. This game is played by two players, I and II on a topological
space $X$. The player I starts the game and chooses a non-empty open set $U_0\subseteq X$. The
player II responds choosing a non-empty open set $V_0\subseteq U_0$. At the $n$th inning the player
I selects a non-empty open set $U_n\subseteq V_{n-1}$ and player II responds by selecting a
non-empty open set $V_n\subseteq U_n$. At the end of the game the player II is declared the winner
if $\bigcap_{n\in\w}U_n\ne\emptyset$. In the opposite case the player I wins the game. By Oxtoby
Theorem \cite[8.11]{Ke}, a topological space $X$ is Baire if and only if the player I has no
winning strategy in the Choquet game on $X$.

Now we recall the definition of a strongly Baire space, introduced by Cao and Moors \cite{CM} via
the game played by two players I and II on a topological space $X$ with a selected dense subset
$D$. The Cao-Moors game has the same moves as the Choquet game and differs merely by the winning
condition. Given a decreasing sequence $U_0\supseteq V_0\supseteq U_1\supseteq V_1\supseteq\cdots$
of open sets constructed by the players I and II, the player II is declared the winner if
$\bigcap_{n\in\w}U_n\ne\emptyset$ and every sequence $(x_n)_{n\in\w}\in\prod_{n\in\w}(V_n\cap D)$ has a cluster point in $X$. A topological space $X$ is called {\em strongly Baire}
if the first player has no winning strategy in the Cao-Moors game for some dense subset $D$ of $X$.

It is easy to see that a metrizable space is Baire if and only if it is strongly Baire.
Strongly Baire spaces are Baire and \v Cech-complete spaces are strongly Baire.

A subset $B$ of a topological space $X$ is called  {\em precompact} if $B$ has the compact closure
$\overline B$ in $X$.

A topological space $X$ is called a {\em $k$-space} if a subset $A\subseteq X$ is closed in $X$ provided for every compact subset $K\subseteq X$ the intersection $A\cap K$ is closed in $K$.  


A topological space $X$ is defined to be
\begin{itemize}
\item {\em first-countably-compact} if for each point $x\in X$ there exists a sequence $(U_n)_{n\in\w}$ of neighborhoods
of $x$ in $X$ such that any sequence $(x_n)_{n\in\w}\in \prod_{n\in\w}U_n$ has a cluster point in $X$;
\item {\em first-countably-pracompact} if for each point $x\in X$ there exist a dense subset $D$ of $X$ and a countable system of neighborhoods $(U_n)_{n\in\w}$ of $x$ in $X$ such that every sequence $(x_n)_{n\in\w}\in \prod_{n\in\w}(U_n\cap D)$ has a cluster point in $X$.
\end{itemize}
The class of first-countably-compact spaces contain all first-countable spaces and all countably
compact spaces. Also it can be shown that each topologically homogeneous strongly Baire space is first-countably-pracompact.

A subset $A$ of a topological space $X$ is called {\em sequentially closed} in $X$ if for any convergent sequence $\{a_n\}_{n\in\w}\subseteq A$ in $X$, the limit point $\lim_{n\to\infty}a_n$ belongs to the set $A$.

\subsection{Co-Namioka spaces}

A topological space $X$ is called {\em co-Namioka} if for any Baire space $Y$ and any separately continuous function $f:X\times Y\to \IR$ there exists a dense $G_\delta$-set $G\subseteq Y$ such that $f$ is continuous at each point of the set $X\times G$.

A continuous map $f:X\to Y$ between topological spaces is {\em perfect} if it is closed and for every $y\in Y$ the preimage $f^{-1}(y)$ is compact.

The proof of the following proposition was suggested by O.V.~Maslyuchenko.

\begin{proposition}\label{p:pcN} Let $f:X\to Z$ be a surjective perfect map between Hausdorff spaces. If the space $X$ is co-Namioka, then so is the space $Z$.
\end{proposition}

\begin{proof} Fix any separately continuous map $\varphi:Z\times Y\to \IR$ and consider the separately continuous function
$$\psi:X\times Y\to \IR,\quad \psi:(x,y)\mapsto \varphi(f(x),y).$$
Since $X$ is co-Namioka, there exists a dense $G_\delta$-set $G\subseteq Y$ such that the function $\psi$ is continuous at each point of the set $X\times G$. We claim that the function $\varphi$ is continuous at each point $(z,g)$ of the set $Z\times G$.

By the continuity of the function $\psi$ at points of the set $X\times \{g\}$, for every $\e>0$ and every $x\in X$ there exists a neighborhood $O_{(x,g)}\subseteq X\times G$ of the point $(x,g)\in X\times G$ such that $\psi(O_{(x,g)})\subseteq (\psi(x,g)-\e,\psi(x,g)+\e)$. We lose no generality assuming that $O_{(x,g)}$ is of basic form $O_x\times G_{x}$ for some open sets $O_x\subseteq X$ and $G_x\subseteq Y$. By the compactness of the set $f^{-1}(z)$ (which follows from the perfectness of $f$), there exists a finite set $F\subseteq f^{-1}(z)$ such that $f^{-1}(z)\subseteq \bigcup_{x\in F}O_x$. Since the map $f$ is closed, the set
$$U=Z\setminus f(X\setminus \bigcup_{x\in F}O_x)$$is an open neighborhood of $z$ in the space $Z$. Consider the open neighborhood $V=\bigcap_{x\in F}G_x$ of $g$ in $Y$. We claim that $\varphi(U\times V)\subseteq(\varphi(z,g)-\e,\varphi(z,g)+\e)$. Indeed, for any $(u,v)\in U\times V$ we have $f^{-1}(u)\subseteq\bigcup_{x\in F}O_x$ and hence there exist $x\in F$ and $w\in O_x$ such that $f(w)=u$. Then
$$
\begin{aligned}
\varphi(u,v)&=\varphi(f(w),v)=\psi(w,v)\in\psi(O_x\times G_x)=\psi(O_{(x,g)})\subseteq(\psi(x,g)-\e,\psi(x,g)+\e)=\\
&=(\varphi(f(x),g)-\e,\varphi(f(x),g)+\e)=(\varphi(z,g)-\e,\varphi(z,g)+\e).
\end{aligned}$$
Therefore, the map $\varphi$ is continuous at points of the set $Z\times G$, witnessing that the space $Z$ is co-Namioka.
\end{proof}

A Hausdorff space $X$ is called
\begin{itemize}
\item {\em dyadic compact} if for some cardinal $\kappa$ there exists a continuous surjective map $f:\{0,1\}^\kappa\to X$;
\item {\em locally dyadic} if each point $x\in X$ has a closed neighborhood which is dyadic compact;
\item {\em Lindel\"of} if each open cover of $X$ contains a countable subcover.
\end{itemize}

\begin{lemma}\label{l:dcN} Every diadyc compact space $X$ is co-Namioka.
\end{lemma}

\begin{proof} Find a cardinal $\kappa$ and a continuous surjective map $f:\{0,1\}^\kappa\to X$. By \cite{Bouziad}, the Cantor cube $\{0,1\}^\kappa$ is co-Namioka and by Proposition~\ref{p:pcN} so is its continuous image $X$.
\end{proof}

\begin{proposition} Every Lindel\"of locally dyadic Hausdorff space $X$ is co-Namioka.
\end{proposition}

\begin{proof}  Let $Y$ be a Baire space and $f:X\times Y\to \IR$ be a separately continuous function. Taking into account that the space $X$ is Lindel\"of and locally dyadic, we can find a countable cover $\K$ of $X$ by dyadic compact subsets whose interiors cover the space $X$. By Lemma~\ref{l:dcN}, every space $K\in\K$
is co-Namioka. So, we can find a dense $G_\delta$-set $G_K$ in $Y$ such that the function $f{\restriction}_{K\times Y}$ is continuous at each point of the set $K\times G_K$. Since the space $Y$ is Baire, the countable intersection $G=\bigcap_{K\in\K}G_K$ is a dense $G_\delta$-set in $Y$. We claim that the function $f$ is continuous at each point of the set $X\times G$. Given any point $(x,y)\in X\times G$, find a dyadic compact set $K\in\K$ containing $x$ in its interior in $X$. Then $K\times Y$ is a neighborhood of the point $(x,y)$ in $X\times Y$. Since $G\subseteq G_K$, the function $f{\restriction}_{K\times Y}$ is continuous at $(x,y)$ and so is the function $f$.
\end{proof}

\subsection{Topological groups} In this subsection we recall some information related to topological groups. 
For two subsets $A,B\subseteq X$ we denote by $AB=\{ab:a\in A,\;b\in B\}$ their pointwise product in $X$. 

For topological groups $X,Y$ we denote by $\Hom_p(X,Y)$ the subspace of the Tychonoff power $Y^X$
consisting of continuous homomorphisms from $X$ to $Y$.

A topological group $X$ is {\em complete} if $X$ is closed in any topological group, containing $X$ as a subgroup. It is well-known \cite[\S3.6]{AT} that each topological group $X$ can be identified with a dense subgroup of a complete topological group $\breve X$ called {\em the Ra\u\i kov completion} of $X$. It can be constructed as the completion of $X$ by the two-sided uniformity (which is generated by the base consisting of the entourages $\{(x,y)\in X\times X: x\in yU\cap Uy\}$ where $U$ is a neighborhood of the
identity in $X$).

The Ra\u\i kov completion has the following extension property:  each continuous homomorphism $h:X\to Y$ to a complete topological group $Y$ uniquely extends to a continuous homomorphism $\breve h:\breve X\to Y$.

A subset $A$ of a topological group $X$ is {\em totally bounded} if for each neighborhood $U$ of
the identity of $X$ there exists a finite subset $F\subseteq X$ such that $A$ is contained in $\bigcup_{x\in
F}xU\cap Ux$. It is known that a subset  of a topological group $X$ is totally bounded if and only
if it has compact closure  in the Ra\u\i kov completion of $X$. In particular, a topological group
is compact if and only if it is complete and totally bounded.

A subset $A$ of a topological group $X$ is {\em left} (resp. {\em right\/}) {\em $\w$-narrow} if for each neighborhood $U$ of the identity of $X$ there exists a countable subset
$C\subseteq X$ such that $A$ is contained in $\bigcup_{x\in C}xU$ (resp. $\bigcup_{x\in C}Ux$).
It is easy to check that a subset $A$ of a topological group $X$ is left $\w$-narrow iff
a subset $A^{-1}$ of $X$ is right $\w$-narrow.
A subset $A$ of a topological group $X$ is {\em $\w$-narrow}, if $A$ is both left and right
$\w$-narrow.

\begin{proposition} Let $A$ and $B$ be left $\w$-narrow subsets of a  topological group $X$. Then
$AB$ is a left $\w$-narrow subset of $X$ too.
\end{proposition}
\begin{proof} Let $U$ be any neighborhood of the identity of $X$. Pick a neighborhood $V$
of the identity of $X$ such that $VV\subseteq U$. There exists a countable subset $D$ of $X$
such that $B\subseteq DV$. For each element $d\in D$ pick a
neighborhood $V_d$ of the identity of $X$ such that $d^{-1}V_dd\subseteq V$.
There exists a countable subset $C_d$ of $X$ such that $A\subseteq C_dV_d$.
Put $C=\left(\bigcup_{d\in D} C_d\right)D$. Let $a\in A$, $b\in B$ be any elements.
There exists $d\in D$ such that $b\in dV$. There exists $c\in C_d$ such that $a\in cV_d$. Then
$$ab\in cV_d\cdot dV=cd\cdot d^{-1}V_ddV\subseteq cdVV\subseteq cdU\subseteq CU.$$
\end{proof}


\begin{corollary}\label{cor:wnarrowunion} For any $\w$-narrow subset $A$ of a  topological group $X$, the subgroup  of $X$ generated by $A$ is  $\w$-narrow.\hfill $\square$
\end{corollary}

\subsection{Bi-homomorphisms}

Let $X,Y,Z$ be topological groups. A function $h:X\times Y\to Z$ is called a {\em bi-homomorphism}
if for every $a\in X$ and $b\in Y$ the functions $${}^a\!h:Y\to Z,\;\;{}^a\!h:y\mapsto h(a,y)\mbox{
\ and \  } h_b:X\to Z,\;\;h_b:x\mapsto h(x,b),$$ are homomorphisms.

A bi-homomorphism $h:X\times Y \to Z$ is called {\em right-continuous} if for any $x\in X$ the homomorphism ${}^x\!h:Y\to Z$, ${}^x\!h:y\mapsto h(x,y)$, is continuous.

\begin{theorem}\label{p2.1} Let $X,Y,Z$ be topological groups and $\breve Y,\breve Z$ be the Ra\u\i kov completions of $X,Y$, respectively. Any right-continuous bi-homomorphism $h:X\times Y\to Z$
uniquely extends to a right-continuous bi-homomorphism $\breve h:X\times \breve Y\to\breve Z$. The
set $\widetilde Y=\{y\in\breve Y:\breve h_y\mbox{ is continuous}\}$ is a subgroup of $\breve Y$. If $X$ is Baire and $\w$-narrow, then the subgroup $\widetilde Y$ is sequentially closed in $\breve Y$. If
$h$ is jointly continuous, then $\widetilde Y=\breve Y$ and the bi-homomorphism $\breve h$ is
jointly continuous.
\end{theorem}

\begin{proof} For every $x\in X$ consider the continuous homomorphism ${}^x\!h:Y\to Z$ and its continuous extension ${}^x\!\breve h:\breve Y\to\breve Z$ to the Ra\u\i kov completions of the groups $Y,Z$.
The homomorphisms ${}^x\!\breve h$, $x\in X$, compose a function $\breve h:X\times \breve Y\to\breve Z$ defined by $\breve h(x,y)={}^x\breve h(y)$ for $(x,y)\in X\times \breve Y$.

Let us check that $\breve h$ is indeed is bi-homomorphism. For any $x\in X$, the map ${}^x\!\breve
h:\breve Y\to\breve Z$ is a homomorphism, being a continuous extension of the homomorphism
${}^x\!h$. Next, we prove that for any point $y\in \breve Y$ the map $\breve h_y:X\to\breve Z$,
$\breve h_y:x\mapsto \breve h(x,y)$, is a homomorphism. Assuming the opposite, we can find two
points $a,b\in X$ such that $\breve h(ab,y)\ne \breve h(a,y)\cdot \breve h(b,y)$. Since the
topological group $\breve Z$ is Hausdorff, we can find disjoint open sets $W,W_+\subseteq\breve Z$
such that $\breve h(ab,y)\in W$ and $\breve h(a,y)\cdot\breve h(b,y)\in W_+$. By the continuity of the multiplication in $\breve Z$, there are open sets
$W_a,W_b\subseteq\breve Z$ such that $\breve h(a,y)\in W_a$, $\breve h(b,y)\in W_b$ and
$W_aW_b\subseteq W_+$. The continuity of the homomorphisms ${}^{ab}\breve h$, ${}^{a}\breve h$,
${}^{b}\breve h$ at $y$ provides a neighborhood $V\subseteq \breve Y$ of $y$ such that
${}^{ab}\breve h(V)\subseteq W$, ${}^{a}\breve h(V)\subseteq W_a$ and ${}^{b}\breve h(V)\subseteq
W_b$. Since $Y$ is dense in $\breve Y$, the neighborhood $V$ contains some point $v\in V\cap Y$.
Then the points $h(ab,v)={}^{ab}\breve h(v)\in W$ and $h(a,v)\cdot h(b,v)={}^a\breve
h(v)\cdot{}^{b}\breve h(v)\in W_aW_b\subseteq W_+$ are distinct, which is not possible as $h_v$ is
a homomorphism.

To see that the set $\widetilde Y$ is a subgroup of $\breve Y$, observe that for any points
$y,y'\in \widetilde Y$ the map $\breve h_{yy'}:X\to \breve Z$ is continuous being the product
$\breve h_y\cdot\breve h_{y'}$ of two continuous maps. Also the map $\breve
h_{y^{-1}}=\mathrm{inv}\circ \breve h_y$ is continuous because of the continuity of   the
inversion $\mathrm{inv}:\breve Z\to\breve Z$, $\mathrm{inv}:z\mapsto z^{-1}$.
\smallskip

Now assuming that $X$ is Baire and $\w$-narrow, we shall prove that the subgroup $\widetilde Y$ in
sequentially closed in $\breve Y$. Fix any sequence $\{y_n\}_{n\in\w}\subseteq\widetilde Y$,
convergent to a point $y\in\breve Y$. For every $n\in\w$, the continuity of the homomorphism
$h_{y_n}:X\to \breve Z$ implies that the image $h_{y_n}(X)$ is an $\w$-narrow subgroup of $\breve
Z$. By Corollary~\ref{cor:wnarrowunion}, the subgroup $H$ of $\breve Z$ generated by the union $\bigcup_{n\in\w}h_{y_n}(X)$ is
$\w$-narrow and so is its closure $\overline H$ in $\breve Z$. Taking into account that
$h_y(x)=\lim_{n\to\infty}h_{y_n}(x)$ for all $x\in X$, we conclude that
$h_y(X)\subseteq\overline{H}$ and hence the subgroup $h_y(X)$ is $\w$-narrow. Now we are ready to
prove that the homomorphism $h_y$ is continuous. Given any neighborhood $U\subseteq \breve Z$ of
the identity, we should find a neighborhood $V$ of the identity in $X$ such that $h_y(V)\subseteq
U$. Choose a neighborhood $W$ of the identity in $\breve Z$ such that $WW^{-1}\subseteq U$. Since the
topological space of $\breve Z$ is Tychonoff, we can additionally assume that the set $W$ is
functionally open in $\breve Z$, which means that $W=\xi^{-1}\big((0,1]\big)$ for some continuous
map $\xi:\breve Z\to[0,1]$. We claim that the set $h_y^{-1}(W)$  is Borel in $X$. This follows from
the representation
$$h_y^{-1}(W)=\bigcup_{n\in\IN}\bigcap_{m=n}^\infty\{x\in X:\xi(h(x,y_m))\ge
\tfrac1n\},$$witnessing that the set $A=h_y^{-1}(W)$ is of type $F_\sigma$ in $X$. By the
$\w$-narrowness of $h_y(X)$, there exists a countable set $C\subseteq X$ such that $h_y(X)\subseteq
h_y(C)\cdot W$. Then $X=C\cdot h_y^{-1}(W)=C\cdot A$. Since $X$ is Baire, the set $A$ is non-meager
in $X$. By Pettis Theorem \cite[9.9]{Ke}, the set $V=AA^{-1}$ is a neighborhood of the identity in
$X$. Then $$h_y(V)=h_y(A\cdot A^{-1})\subseteq h_y(A)\cdot h_y(A)^{-1}\subseteq  WW^{-1}\subseteq U,$$
witnessing that the homomorphism $h_y$ is continuous.
\smallskip

Finally, assuming the continuity of the bi-homomorphism $h$, we shall prove that the
bi-ho\-mo\-mor\-phism $\breve h$ is continuous. The separate continuity of $h$ implies that
$Y\subseteq \widetilde Y$, so $\widetilde Y$ is dense in $\breve Y$. First we prove that $\breve h$
is continuous at the identity. Given any neighborhood $W\subseteq\breve Z$ of the identity, find
a neighborhood $W_0\subseteq Z$ of the identity such that $\overline{W}_0\subseteq W$. Using the continuity
of the homomorphism $h:X\times Y\to Z$,  find a neighborhood $U$ of the identity in $X$ and a neighborhood
$V$ of the identity in $\breve Y$ such that $h(V\times (U\cap Y))\subseteq W_0$. For every $v\in V$ the
continuity of the homomorphism ${}^v\breve h$ implies that ${}^v\breve h(U)\subseteq{}^v\breve
h(\overline{U\cap Y})\subseteq\overline{W}_0\subseteq W$. Consequently, $\breve h(V\times
U)\subseteq W$ and the bi-homomorphism $\breve h:X\times \breve Y\to\breve Z$ is continuous at the identity and hence continuous everywherem see \cite[3.1]{EV}.
\end{proof}

Now we present two theorems that will help us to establish the joint continuity of some bi-homomorphisms. The first of them was proved by Cao and Moors in \cite{CM}.

\begin{theorem}[Cao, Moors]\label{p2.9} Let $T$ be a topological space, $X,Y$ be topological groups and $h:T\times X\to Y$ be a separately continuous map such that for every $t\in T$ the function ${}^th:X\to Y$, ${}^th:x\mapsto h(t,x)$, is a homomorphism.
If the space $X$ is strongly Baire, then for every first-countably-pracompact subspace $K\subseteq T$ the restriction $h{\restriction}_{K\times X}$ is jointly continuous.
\end{theorem}


\begin{theorem}\label{t:2.6} Let $T$ be a topological space, $X,Y$ be topological groups and $h:T\times X\to Y$ be a separately continuous function such that for every $t\in T$ the function ${}^th:X\to Y$, ${}^th:x\mapsto h(t,x)$, is a homomorphism.
If the space $X$ is Baire, then for every co-Namioka compact subspace $K\subseteq T$ the restriction $h{\restriction}_{K\times X}$ is jointly continuous.
\end{theorem}

\begin{proof} Assume that the space $X$ is Baire and $K$ is a compact co-Namioka subspace of $T$. We claim that the map $h{\restriction}_{K\times X}$ is continuous. To derive a contradiction, assume that $h{\restriction}_{K\times X}$ has a discontinuity point $(z,x)\in K\times X$. Then there exists a neighborhood
$U$ of the identity $e$ in the topological group $Y$ such that $h(V)\not\subseteq h(z,x)\cdot U$ for any
neighborhood $V$ of $(z,x)$ in $K\times X$. Find a neighborhood of $W$ of $e$ in $Y$ such that
$WW^{-1}W\subseteq U$.

By \cite[3.3.9]{AT}, there exists a left-invariant continuous pseudometric $\rho:Y\times Y\to \IR$ such that $B_1\subseteq W$ where $B_1=\{y\in Y:\rho(e,y)<1\}$. Let $Y_\rho$ be the space $Y$ endowed with the pseudometric $\rho$ and $h_\rho$ be the map $h$ considered as a map with values in the pseudometric space $Y_\rho$.

Since the compact space $K$ is  co-Namioka, we can apply \cite[3.1]{Bouziad} and find a dense $G_\delta$-set
$G\subseteq X$ such that the map $h_\rho{\restriction}_{K\times X}$ is continuous at each point of
the set $K\times G$. Fix any point $g\in G$ and by the continuity of the map
$h_\rho{\restriction}_{K\times X}$ at $(z,g)$, find open sets $U_z\subseteq K$ and $V_g\subseteq X$
such that $(z,g)\in U_z\times V_g$ and $h_\rho(U_z\times V_g)\subseteq  h(z,g)\cdot B_1\subseteq
h(z,g)\cdot W$. Replacing $U_z$ by a smaller neighborhood, if needed, we can additionally assume
that $h(U_z\times\{x\})\subseteq h(z,x)\cdot W$ and $h(U_z\times\{g\})\subseteq h(z,g)\cdot W$.

 Observe that the set $xg^{-1}V_g$ is a neighborhood of the point $x$ in the topological group $X$.
 By the choice of the set $U$, there is a point $(u,v)\in U_z\times V_g$ such that $h(u,xg^{-1}v)\notin h(z,x){\cdot} U$. On the other hand,
 \begin{multline*}h(u,xg^{-1}v)=h(u,x)h(u,g)^{-1}h(u,v)\subseteq h(z,x)\cdot WW^{-1}h(z,g)^{-1}\cdot h(z,g)\cdot W=\\
 =h(z,x)WW^{-1}W\subseteq h(z,x){\cdot}U,
\end{multline*}
which is a required contradiction.
\end{proof}

\subsection{Sep-joint automatic pairs of topological groups}

\begin{definition} A pair $(X,Y)$ of topological groups is called {\em sep-joint automatic} if every separately continuous bi-homomorphism $h:X\times Y\to Z$ into a topological group $Z$ is jointly continuous.
\end{definition}

\begin{theorem}\label{t2.8} A pair $(X,Y)$ of topological groups is sep-joint automatic if one of the following conditions holds:
\begin{enumerate}
\item the pair $(Y,X)$ is sep-joint authomatic;
\item $X$ is strongly Baire and $Y$ is first-countably-pracompact;
\item $X$ is locally compact and $Y$ is Baire;
\item $X$ is locally compact and $Y$ is a $k$-space.
\end{enumerate}
\end{theorem}

\begin{proof} 1, 2. The first statement is trivial and the second one follows from Theorem~\ref{p2.9}.
\smallskip

3. To prove the third statement, assume that the topological group $X$ is locally compact and $Y$ is Baire. Take any separately continuous bi-homomorphism $h:X\times Y\to Z$ to a topological group $Z$. To prove that $h$ is continuous, take any pair $(x,y)\in X\times Y$. By Theorem 3.1.15 in \cite{AT}, the locally compact topological group $X$ is locally dyadic compact. Consequently, we can find a dyadic compact space $K\subseteq X$ containing $x$ in its interior in $X$. By Lemma~\ref{l:dcN}, the dyadic compact space $K$ is co-Namioka and by Theorem~\ref{t:2.6}, the restriction $h{\restriction}_{K\times Y}$ is jointly continuous. Since $K\times Y$ is a neighborhood of $(x,y)$, the bi-homomorphism $h$ is continuous at $(x,y)$.
\smallskip

4. Assume that $X$ is locally compact and $Y$ is a $k$-space. By  \cite[3.3.27]{En}, the product
$X\times Y$ is a $k$-space. To prove that the pair $(X,Y)$ is sep-joint automatic, take any
separately continuous bi-homomorphism $h:X\times Y\to Z$ to a topological group $Z$. Since $X\times
Y$ is a $k$-space, it suffices to prove that for every compact space $K\subseteq X\times Y$, the
restriction $h{\restriction}_K$ is continuous. Find a compact set $K_Y\subseteq Y$ such that
$K\subseteq X\times K_Y$. Taking into account that the locally compact space $X$ is strongly Baire
and the compact space $K_Y$ is first-countably-pracompact, we can apply Theorem~\ref{p2.9} and
conclude that the restriction $h{\restriction}_{X\times K_Y}$ is continuous and so is the
restriction $h{\restriction}_K$.
\end{proof}

\begin{remark} In \cite{MP}  Mykhaylyuk and Pol constructed a compact Hausdorff space $K$ whose function space $F=C_p(K,\{-1,1\})$ is Baire in the topology of pointwise convergence. For this compact space $K$, the evaluation map $e:K\times F\to \{-1,1\}$, $e:(x,f)\mapsto f(x)$, is separately continuous and everywhere discontinuous. This example shows that ``strong Baire'' in Cao--Moors Theorem~\ref{p2.9} cannot be weakened to ``Baire" and also that Theorem~\ref{t2.8} cannot be extended beyond the class of bi-homomorphisms.
\end{remark}

\section{Semitopological modules and their completions}\label{s3}

In this section we apply the results of the preceding section to describe the structure of
semitopological modules over topological rings and to study the continuity properties of the
multiplication map in semitopological modules.

By a {\em topological ring} we understand a unital ring $R$ endowed with a Hausdorff topology
$\tau$ making the addition $+:R\times R\to R$ and multiplication $\cdot:R\times R\to R$ continuous.
A ring $R$ will be called {\em unital} if it has a unit (i.e., an element $1\in R$ such that
$1x=x1=x$ for all $x\in X$). An element $a$ of a ring $R$ is called {\em invertible} if there exists an element $a^{-1}\in R$ such that $aa^{-1}=a^{-1}a$ is the unit of $R$. It follows that a ring is unital if and only if it has an invertible element. A commutative ring is a {\em field} if every non-zero element of the ring is invertible.

Theorem~\ref{p2.1} implies that the Ra\u\i kov completion $\breve R$ of (the  additive group) of a topological ring $R$ has
the structure of a topological ring. This topological ring will be called the {\em completion} of $R$.

A {\em module over a ring $R$} (briefly, an {\em $R$-module}) is an Abelian group $X$ endowed with
a map $*:R\times X\to X$, $*:(a,x)\mapsto a*x$, such that for every $a,b\in R$ and $x,y\in X$ the
following conditions are satisfied:
\begin{itemize}
\item $a*(x+y)=a*x+a*y$;
\item $(a+b)* x=a*x+b*x$;
\item $(ab)* x=a* (b* x)$;
\item $1* x=x$.
\end{itemize}
The first two conditions imply that the multiplication map $*:R\times X\to X$ is a bi-homomorphism.
So, for points $r\in R$ and $b\in X$ the functions ${}^r*:X\to X$, ${}^r*:x\mapsto r*x$, and
$*^b:R\to X$, $*^b:\rho\mapsto\rho*b$, are homomorphisms.

A {\em right-topological module over a topological ring $R$} (briefly, a {\em right-topological
$R$-module}) is a module $X$ over $R$ endowed with a Hausdorff topology $\tau$ turning $X$ into an
Abelian topological group and making the multiplication map $*:R\times X\to X$ right-continuous. If
the multiplication map $*$ is separately continuous, then $(X,\tau)$ is called a {\em
semitopological $R$-module}; if $*$ is jointly continuous, then $(X,\tau)$ is called a {\em
topological $R$-module}.

Right-topological $R$-modules are objects of the category whose morphisms are continuous $R$-linear maps between right-topological $R$-modules.
A map $h:X\to Y$ between two $R$-modules is called {\em $R$-linear} if $$h(x_1+x_2)=h(x_1)+h(x_2)\quad\mbox{and}\quad h(a*x_1)=a*h(x_2)$$ for any points $x_1,x_2\in X$ and $a\in R$.

It is clear that (semi)topological linear spaces over topological fields are partial cases of (semi)\-topological modules over topological rings.

Given a right-topological module $X$ over a topological ring $R$, consider the multiplication map
$*:R\times X\to X$, $*:(a,x)\mapsto a*x$, and observe that it is a right-continuous
bi-homomorphism. By Theorem~\ref{p2.1}, this bi-homomorphism can be uniquely extended to a
right-continuous bi-homomorphism $\breve *:R\times\breve X\to\breve X$, which turns $\breve X$ into
a right-topological $R$-module, which will be called the {\em Ra\u\i kov completion} of $X$.

For a right-topological $R$-module $X$ the subgroup $$\widetilde X=\{x\in\breve
X:\breve*_x:R\to \breve X\mbox{ is continuous}\}$$is an $R$-submodule of $\breve X$. Indeed, for any
$x\in\widetilde X$ and $a\in R$ the continuity of the map $*_{a*x}:R\to \breve X$ follows from the
equality $*_{a*x}(r)=r*(a*x)=(ra)*x=*_x(ra)$, $r\in R$, and the continuity of the multiplication in
the topological ring $R$. By definition, $\widetilde X$ is the maximal semitopological
$R$-submodule of $\breve X$. If $X$ is a semitopological $R$-module, then the submodule $\widetilde
X$ contains $X$ and is called the {\em semitopological $R$-completion} of $X$.

The construction of Ra\u\i kov completion determines a functor in the category of right-topological
$R$-modules.

Indeed, any continuous $R$-linear map $h:X\to Y$ between semitopological $R$-modules can be
extended to a continuous homomorphism $\breve h:\breve X\to\breve Y$. For every $a\in R$, the
continuity of multiplication by $a$ in $X$ and $Y$, and the continuity of the
map $\breve h$ imply that the set $\{x\in\breve X:\breve h(a*x)=a* \breve h(x)\}\supseteq X$ is
closed in $\breve X$ and being dense coincides with $\breve X$. This means that the homomorphism
$\breve h$ is $R$-linear.

The construction of semitopological $R$-completion also is functorial in the
category of semitopological $R$-modules. Indeed, for any continuous $R$-linear map $h:X\to Y$
between semitopological $R$-modules, consider its continuous extension $\breve h:\breve X\to\breve
Y$ to the Ra\u\i kov completions of $X$ and $Y$. Observe that for any $x\in\widetilde X$ the
continuity of the homomorphism $*_x:R\to\widetilde X\subseteq\breve X$, $*_x:a\mapsto a*x$, implies
the continuity of the homomorphism $*_{\breve h(x)}=\breve h\circ*_x:R\to \breve Y$, $*_{\breve
h(x)}:a\mapsto a*\breve h(x)=\breve h(a*x)$. So $\breve h(x)\in \widetilde Y$ and $\breve
h(\widetilde X)\subseteq \widetilde Y$. Then the restriction $\tilde h=\breve h{\restriction}_{\widetilde
X}:\widetilde X\to\widetilde Y$ is a well-defined continuous $R$-linear map between the
semitopological $R$-modules $\widetilde X$ and $\widetilde Y$.

\begin{theorem}\label{t3.1} For any semitopological $R$-module $X$ over a topological ring $R$, its $R$-completion $\widetilde X$ is a semitopological $R$-module having the following properties.
\begin{enumerate}
\item The semitopological $R$-module $\widetilde X$ is closed in any semitopological $R$-module $Y$ containing $\widetilde X$;
\item If $X$ is a topological $R$-module, then $\widetilde X$ coincides with the Ra\u\i kov completion $\breve X$ of $X$ and is a topological $R$-module.
\item If the additive topological group of $R$ is Baire and $\w$-narrow, then the set $\widetilde X$ is sequentially closed in the Ra\u\i kov completion $\breve X$ of $X$.
\end{enumerate}
\end{theorem}

\begin{proof} 1. Let $Y$ be any semitopological $R$-module containing the semitopological $R$-module $\widetilde X$. Since $X\subseteq\widetilde X\subseteq Y$, the Ra\u\i kov completion $\breve Y$ of $Y$ contains the Ra\u\i kov completion $\breve X$ of $X$ and $\breve X$ is closed in $\breve Y$ (see \cite[3.6.18]{AT}). For every $y\in Y\cap\breve X$ the continuity of the map $*_y:R\to Y\cap \breve X$, $*_y:a\mapsto a*y$, implies that $y\in \widetilde X$. Now we see that the set $\widetilde X=\breve X\cap Y$ is closed in $Y$.
\smallskip

2. If $X$ is a topological $R$-module, then by Theorem~\ref{p2.1}, the continuity of the multiplication map $*:R\times X\to X$ implies the continuity of its extension $\breve h:R\times \breve X\to \breve X$, which means that $\breve X$ is a topological $R$-module and $\widetilde X=\breve X$.
\smallskip

3. If the additive topological group of $R$ is Baire and $\w$-narrow, then the set $\widetilde X$ is sequentially closed in $\breve X$ according to Theorem~\ref{p2.1}.
\end{proof}

Let $X$ be a right-topological $R$-module over a topological ring $R$. We say that a subset
$B\subseteq X$ is {\em $R$-bounded} if for every neighborhood $U\subseteq X$ of zero there is a
neighborhood $V\subseteq R$ of zero such that $V*B\subseteq U$. This is equivalent to saying that
the family of homomorphisms $\{*_x:R\to X\}_{x\in B}\subseteq \Hom_p(R,X)$ is equicontinuous. This
simple observation allows us to apply Theorem~\ref{t2.8} and obtain the following  theorems.

\begin{theorem}\label{t3.2} Let $X$ be a semitopological $R$-module over a topological ring $R$.
\begin{enumerate}
\item For any $R$-bounded set $B\subseteq X$ the multiplication map $R\times B\to X$, $(r,x)\mapsto r*x$, is continuous.
\item If the space $R$ is strongly Baire, then each compact subset of $X$ is $R$-bounded.
\end{enumerate}
\end{theorem}

\begin{theorem}\label{t3.4} A semitopological $R$-module $X$ is a topological $R$-module if $(R,X)$ is a sep-joint automatic pair of topological groups. This happens if the topological ring $R$ is locally compact and the space $X$ is Baire or a $k$-space.
\end{theorem}

\section{The Bohr modifications of semitopological modules}\label{s4}

In this section we shall answer Problem~\ref{prob} by constructing (many) examples of semitopological linear spaces which are not topological linear spaces. These examples are just usual linear topological spaces endowed with the Bohr topology.

The {\em Bohr topology} on a topological group $X$ is the largest totally bounded group topology on $X$. It can be equivalently defined as the smallest topology $\tau$ on $X$ such that any continuous homomorphism $h:X\to K$ to a Hausdorff compact topological group $K$ remains continuous with respect to the topology $\tau$. By $X^{\flat}$ we shall denote the group $X$ endowed with the Bohr topology. This topological group will be called the {\em Bohr modification} of $X$. 

A topological group $X$ is called {\em Bohr separated} if the Bohr topology on $X$ is Hausdorff. For example, each locally convex linear topological space is Bohr separated. On the other hand, the linear metric spaces $L_p$, $0<p<1$, are not Bohr separated, see \cite[4.2.3]{Rol}.

Observe that any continuous homomorphism $h:X\to Y$ between topological groups remains continuous as a homomorphism $h:X^{\flat}\to Y^{\flat}$ between the Bohr modifications of $X$ and $Y$. This observation implies the following simple fact.

\begin{proposition} For any Bohr separated semitopological module $X$ over a topological ring $R$ the Bohr modification $X^{\flat}$  is a semitopological $R$-module.
\end{proposition}

By definition, the Bohr modification $X^{\flat}$ of any topological group $X$ is totally bounded. It turns out that the Bohr modification of a semitopological module is multiplicatively bounded in the following sense.

\begin{proposition}\label{p4.2} Let $X$ be a semitopological module over a topological ring $R$. For every neighborhood $U\subseteq X^{\flat}$ of zero there exists a number $n\in\IN$ such that for any subset $A\subseteq R$ of cardinality $|A|\ge n$ and any $x\in X$ there are two distinct points $a,b\in A$ such that $(a-b)*x\in U$.
\end{proposition}

\begin{proof} By the definition of the Bohr topology on $X$, there exist a continuous homomorphism $h:X\to H$ to a compact Abelian topological group and a neighborhood $W\subseteq H$ of the neutral element such that $h^{-1}(W-W)\subseteq U$. Replacing $H$ by the closure of $h(X)$, we can assume that $h(X)$ is dense in $H$. Let $\mu$ be the Haar measure on $H$ and $n\in\IN$ be any number such that $\frac1n<\mu(W)$. Let $A\subseteq R$ be any subset of cardinality $|A|\ge n$. For any points $x\in X$ and $a\in A$, consider the point $a{*}x$ and the set $h(a{*}x)+W\subseteq H$. Observe that $\mu(h(a{*}x)+W)=\mu(W)>\frac1{n}\ge\frac1{|A|}$, which implies that the indexed family $\big(h(a{*}x)+W)_{a\in A}$ is not disjoint and hence there are two distinct points $a,b\in A$ such that $(h(a{*}x)+W)\cap (h(b{*}x)+W)\ne\emptyset$. Then $h((a-b)*x)=h(a{*}x)-h(b{*}x)\in W-W$ and hence $(a-b)*x\in h^{-1}(W-W)\subseteq U$.
\end{proof}

We shall use Proposition~\ref{p4.2} to detect topological rings $R$ such that for any non-trivial semitopological $R$-module $X$ its Bohr modification $X^{\flat}$ is not a topological $R$-module.

\begin{definition} A topological ring $R$ is called
\begin{itemize}
\item {\em crowded} if every neighborhood of zero in $R$ contains an invertible element;
\item {\em overcrowded} if $R$ is crowded and for every $n\in\IN$ there is a set $A\subseteq X$ of cardinality $|A|=n$ such that for any distinct elements $a,b\in A$ the element $a-b$ is invertible in $R$.
\end{itemize}
\end{definition}

 It is clear that every crowded topological ring is not discrete; a topological field $R$ is overcrowded if and only if it is crowded if and only if $R$ is not discrete. 

\begin{theorem}\label{t4.4} Let $X$ be a Bohr separated semitopological module over a topological ring $R$. If the completion $\breve R$ of $R$ is a crowded topological ring, then the Bohr modification $X^{\flat}$ of $X$ is not a topological $R$-module. Moreover, for any non-zero element $x\in X$ the homomorphism $*_x:R\to X^{\flat}$, $*_x:r\mapsto r* x$, is not a topological embedding.
\end{theorem}

\begin{proof} Given any non-zero element $x\in X$, we shall prove that the restriction $*{\restriction}_{R\times R{*}x}$ of the multiplication map $*:R\times X^{\flat}\to X^{\flat}$ is discontinuous (here $R{*}x=\{r{*}x:r\in R\}$). To derive a contradiction, assume that this restriction is continuous. By Theorem~\ref{p2.1}, the continuous bi-homomorphism $*{\restriction}_{R\times R{*}x}$ extends to a continuous bi-homomorphism $\breve *:\breve R\times R{*}x\to \overline{X^\flat}$, where $\overline{X^\flat}$ is the Ra\u\i kov-completion of the totally bounded topological group $X^\flat$.
Consider the subset $\breve Rx=\{r\breve *x:r\in\breve R\}$ and observe that $\breve  Rx$ is contained in the closure $\overline{R{*}x}$ of $R{*}x$ in $\overline{X^\flat}$. Applying Theorem~\ref{p2.1} once more, find a continuous bi-homomorphism $\star:\breve R\times\breve Rx\to \overline{X^\flat}$ extending the bi-homomorphism $\breve *$. Using the continuity of $\star$ it is easy to show that $\star(\breve R\times \breve Rx)\subseteq\breve Rx$ and $\breve Rx$ is a topological $\breve R$-module.

Since the group $\breve Rx$ is Hausdorff, the set $\breve Rx\setminus\{x\}$ is an open
neighborhood of zero in $\breve Rx$. By the continuity of the operation $\star$, there exists a
neighborhood $V$ of zero in $\breve R$ and a neighborhood $W$ of zero in $\breve Rx$ such that
$V\star W\subseteq \breve Rx\setminus\{x\}$. Since the totally bounded topological group $\breve Rx\subseteq \overline{X^\flat}$
coincides with its Bohr modification, $W$ is a neighborhood of zero in the Bohr topology of the
group $\breve Rx$. By Proposition~\ref{p4.2}, there in a number $n\in\IN$ such that for any set
$A\subseteq \breve R$ of cardinality $|A|=n$ there are two distinct points $a,b\in A$ such that
$(a-b)\star x\in W$. The  ring $\breve R$, being overcrowded,  contains a set $B\subseteq \breve R$
of cardinality $|B|=n$ such that the set $D=\{a-b:a,b\in B,\;a\ne b\}$ consists of invertible elements of the ring $\breve R$. It follows that each element $a\in D$ has the inverse $a^{-1}$ in $R$ and hence the set $D^{-1}=\{a^{-1}:a\in D\}$ is well-defined. Using the continuity of multiplication at zero in the topological
ring $\breve R$, find a neighborhood $V_0\subseteq\breve R$ of zero such that $D^{-1}V_0\subseteq
V$. Since $\breve R$ is (over)crowded, the neighborhood $V_0$ contains an invertible element $c$.
By the choice of the number $n$, the $n$-element set $c^{-1}B$ contains two distinct elements
$c^{-1}a,c^{-1}b$ such that $(c^{-1}a-c^{-1}b)\star x\in W$. Then
$x=(a-b)^{-1}c(c^{-1}a-c^{-1}b)\star x\in D^{-1}V_0\star W\subseteq V\star W\subseteq \breve
Rx\setminus\{x\}$, which is a required contradiction completing the proof of the discontinuity of
the restriction $*{\restriction}_{R\times R{*}x}$.
\smallskip

Now we show that the map $*_x:R\to X^{\flat}$ is not a topological embedding. Assuming the opposite, we conclude that $*_x:R\to R{*}x$ is a homeomorphism and then $*{\restriction}_{R\times R{*}x}$ is continuous as a composition of continuous maps (since $a*(b*x)=(a{\cdot} b)*x=*_x(a\cdot *_x^{-1}(bx))$ for any $a,b\in R$). But this contradicts the discontinuity of $*{\restriction}_{R\times R{*}x}$ proved above.
\end{proof}

Since the fields $\IR$ and $\IC$ are overcrowded, Theorem~\ref{t4.4} implies the following corollary answering Problem~\ref{prob}.

\begin{corollary} For any non-trivial semitopological linear space $X$ over a topological field $R\in\{\IR,\IC\}$, the Bohr modification $X^{\flat}$ is not a topological linear space and for any element $x\in X$ the map $*_x:R\to X^{\flat}$, $*_x:\lambda\to\lambda\cdot x$, is not a topological embedding.
\end{corollary}

Next, we construct an example of a semitopological $R$-module $X$ such that $X$ is not a topological $R$-module but for every non-zero element $x\in X$ the map $*_x:R\to X$, $*_x:r\mapsto r*x$, is a topological embedding.

Given a Bohr separated semitopological $R$-module $X$ over a topological ring $R$, consider the $R$-submodule $X^{<\w}=\{(x_n)_{n\in\w}\in X^\w:\exists n\in\w\;\forall m\ge n\; x_m=0\}$ of the countable power $X^\w$. Endow  $X^{<\w}$ with the group topology $\tau$ whose base at zero consists of the sets $(V,W)_n=X^{<\w}\cap (V^n\times W^{\w\setminus n})$ where $n\in\w$, $V$ is a neighborhood of zero in $X$ and $W$ is a neighborhood of zero in $X^{\flat}$. It is clear that the topology $\tau$ turns $X^{<\w}$ into a Bohr separated semitopological $R$-module, which will be denoted by $X^{\natural\w}$.

The following example yields a more refined counterexample to Problem~\ref{prob}.

\begin{example} Let $R$ be a topological ring with overcrowded  completion $\breve R$. For any non-trivial Bohr separated topological $R$-module $X$, the Bohr separated semitopological $R$-module $X^{\natural\w}$ is not a topological $R$-module but any finitely generated $R$-submodule of $X^{\natural\w}$ is a topological $R$-module.
\end{example}

We shall apply Theorem~\ref{t4.4} to construct an example of a semitopological $R$-module $X$ whose
semitopological $R$-completion $\widetilde X$ is strictly smaller than its Ra\u\i kov completion
$\breve X$.

\begin{proposition} Let $R$ be a Baire topological ring whose Ra\u\i kov completion $\breve R$ is overcrowded. For any totally bounded semitopological $R$-module $X$ the $R$-completion $\widetilde X$ of $X$ is strictly smaller than the Ra\u\i kov completion $\breve X$ of $X$.
\end{proposition}

\begin{proof} The total boundedness of $X$ implies that $X=X^\flat$ and the Ra\u\i kov completion $\breve X$ of $X$ is compact. Assuming that $\widetilde X=\breve X$, we conclude that $\breve X$ is a compact semitopological $R$-module. By Theorem~\ref{t2.8}, $\breve X$ is a topological $R$-module. But this contradicts Theorem~\ref{t4.4}.
\end{proof}

In light of Theorem~\ref{t4.4}, it is important to detect topological rings with overcrowded Ra\u\i
kov completions. We shall show that the class of such topological rings includes normable
topological rings with crowded Ra\u\i kov completions.

A topological ring $R$ is {\em normable} if its topology is generated by some norm $|\cdot|:R\to \IR$. A function $|\cdot|:R\to\IR$ on a ring $R$ is called a {\em norm} if for any $x,y\in R$ the following conditions are satisfied:
\begin{itemize}
\item $|x|\ge 0$ and $|x|=0$ iff $x=0$;
\item $|x+y|\le|x|+|y|$;
\item $|xy|=|x|\cdot|y|$.
\end{itemize}
Such norm $|\cdot|$ generates the invariant metric $d:R\times R\to\IR$, $d:(x,y)\mapsto|x-y|$ on $R$, which generates a topology turning $R$ into a topological ring.

\begin{proposition}\label{p4.8} A complete normed topological ring $R$ is overcrowded if and only if $R$ is crowded.
\end{proposition}

\begin{proof} Assume that $R$ is crowded and fix a norm $|\cdot|:R\to\IR$ generating the topology of $R$. By a standard argument it can be shown that for any element $x\in R$ of norm $|x|<1$, the element $1-x$ is invertible in $R=\breve R$ and $(1-x)^{-1}=\sum_{n=0}^\infty x^n$. Since $R$ is crowded, it contains an invertible element $x$ with norm $0<|x|<1$. Consider the set $A=\{x^n:n\in\IN\}$ and observe that for any numbers $n<m$
the difference $x^n-x^m=x^n(1-x^{m-n})$ is invertible in $R$.
\end{proof}

Both completeness and normability of the ring $R$ in Proposition~\ref{p4.8} is essential as shown by the following examples.

\begin{example}  The smallest subring $R$ of $\IR$ that contains the transcendent numbers $\pi$ and $\pi^{-1}$ is crowded but not overcrowded.
\end{example}

We recall that an {\em ordered field} is a field $F$ endowed with a linear order such that for any elements $x,y,a\in F$ the inequality $x<y$ implies $x+a<y+a$ and the inequalities $0<x$, $0<y$ imply $0<xy$.
It is known that an ordered field has characteristic zero and hence contains a subfield isomorphic to the field $\IQ$ of rational numbers.
Each ordered field $F$ carries the order topology generated by open order intervals.

\begin{example} There exists a subring $R$ of a countable ordered field $F$ such that $R$ is crowded but the completion $\breve R$ is not overcrowded.
\end{example}

\begin{proof} Let $x_0=1$ and for every $n\in\IN$ choose a positive real number $x_n$ which does not belong to the smallest algebraically closed field $F_n\subseteq \IC$ containing the set $\{x_i\}_{i<n}$. Consider the real field $F=\bigcup_{n\in\w}(F_n\cap\IR)$. It is easy to see that the field $F$ can be ordered so that $0<x_n<y$ for all $n\in\IN$ and $y\in F_n\cap \IR$. Endow the field $F$ with the topology generated by the linear order. Let $R$ be the smallest subring of $F$ containing the set $\{x_n,x_n^{-1}:n\in\w\}$. It can be shown that $R$ is crowded and its completion $\breve R$ is not overcrowded.
\end{proof}

\section{Bornomodifications of semitopological modules}\label{s5}

In this section given a topological ring $R$ and a family $\cB$ of bounded subsets on $R$, we discuss the (functorial) construction of  $\cB$-modification acting in the category of semitopological $R$-modules.

Let $R$ be a topological ring. A subset $B\subseteq R$ is called {\em bounded} in $R$ if for any neighborhood $U\subseteq R$ of zero there is a neighborhood $V\subseteq R$ of zero such that $B\cdot V\subseteq U$. So, $B$ is $R$-bounded with respect to the right action of $R$ on $R$.

Observe that for any bounded sets $A,B$ in a topological ring $R$ the product $AB$ is  bounded in $R$. Indeed, for any neighborhood $U\subseteq R$ of zero the boundedness of $A$ yields a neighborhood $V\subseteq R$ of zero such that $AV\subseteq U$ and the boundedness of $B$ yields a neighborhood $W\subseteq R$ of zero such that $BW\subseteq V$. Then $(AB)W=A(BW)\subseteq AV\subseteq U$, witnessing that the set $AB$ is bounded in $R$.

\begin{definition}
A family $\mathcal B$ of bounded subsets of a topological ring $R$ is called a {\em bornology} on $R$ if
\begin{enumerate}
\item for any set $B\in\mathcal B$, any subset of $B$ belongs to $\mathcal B$;
\item for any sets $A,B\in\mathcal B$ the union $A\cup B$ belongs to $\mathcal B$;
\item for any set $B\in\mathcal B$ and element $r\in R$ the set $B\cdot r$ belongs to $\mathcal B$;
\item $\bigcup\mathcal B=R$.
\end{enumerate}
\end{definition}

A topological ring $R$ is defined to be {\em locally $\cB$-bounded} for a bornology $\cB$ on $R$ if
for any set $B\in\cB$ there is a neighborhood $V\subseteq R$ of zero such that $BV\in\cB$.  A
topological ring is {\em locally  bounded} if $R$ is locally $\M$-bounded for the maximal bornology
$\M$ consisting of all bounded subsets of $R$.  It is easy to show that a 
topological ring is locally bounded iff it has a bounded neighborhood of the zero.
\smallskip

Let $X$ be a semitopological module over a topological ring $R$. For two subsets $A,B\subseteq X$ we write $A\Subset B$ if $A+U\subseteq B$ for some neighborhood $U\subseteq X$ of zero in $X$. For any subsets $U\subseteq X$ and $B\subseteq R$ consider the set $$U/B=\{x\in U:B*x\Subset U\}.$$

It is clear that for any sets $U\subseteq V\subseteq X$ and $A\subseteq B\subseteq R$ we have $U/B\subseteq V/A$. Consequently, for any sets $U,V\subseteq X$ and $A,B\subseteq R$ we have $(U\cap V)/(A\cup B)\subseteq (U/A)\cap(V/B)$.

Let $\mathcal \cB$ be a bornology on a topological ring $R$. For any semitopological $R$-module $X$ denote by  $\tau^{\sharp\cB}$ the family of all sets $W\subseteq X$ such that for every point $w\in W$ there exist a neighborhood $U$ of zero in $X$ and a bounded set $B\in\cB$ such that $w+U/B\subseteq W$. It is clear that $\tau^{\sharp\cB}$ is a translation invariant topology on $X$ containing the original topology of $X$.

The $R$-module $X$ endowed with the topology $\tau^{\sharp\cB}$ will be denoted by  $X^{\sharp\cB}$ and called the {\em $\cB$-modification} of $X$.

\begin{proposition} Let $\mathcal B$ be a bornology on a topological ring $R$. For any semitopologial $R$-module $X$ its
$\cB$-modification $X^{\sharp\cB}$ is a semitopological $R$-module.
\end{proposition}

\begin{proof} First we check that for any open neighborhood $U\subseteq X$ of zero and any bounded set $B\in\cB$, the set $U/B$ belongs to the topology $\tau^{\sharp\cB}$. Fix any point $x\in U/B\subseteq U$ and find a neighborhood $V\subseteq U$ of zero such that $V+B*x\subseteq U$. We claim that $x+V/B\subseteq U/B$. Indeed, for any $y\in V/B$ we can find a neighborhood $W\subseteq V$ of zero such that $W+B*y\subseteq V$. Then $W+B*(x+y)\subseteq W+B*y+B*x\subseteq V+B*x\subseteq U$, which yields the desired inclusion $x+y\in U/B$. So, the set $U/B$ belongs to the topology $\tau^{\sharp\cB}$.

Now we can prove that the addition operation $+:X^{\sharp\cB}\times X^{\sharp\cB}\to X^{\sharp\cB}$
is continuous. Since the topology $\tau^{\sharp\cB}$ is invariant under translations, it suffices
to prove that the addition is continuous at zero. Fix any neighborhood $W\in\tau^{\sharp\cB}$ of
zero and find an open neighborhood $U\subseteq X$ of zero and a bounded set $B\in\cB$ such that
$U/B\subseteq W$. By the continuity of the addition in $X$, there is an open neighborhood
$V\subseteq X$ of zero such that $V+V\subseteq U$. We claim that the open neighborhood
$V/B\in\tau^{\sharp\cB}$ has the property $V/B+V/B\subseteq U/B\subseteq W$. Given any points
$x,y\in V/B$, find a neighborhood $V_0\subseteq X$ of zero such that $V_0+B*x\subseteq V$ and
$V_0+B*y\subseteq V$. Then $V_0+B*(x+y)\subseteq V_0+V_0+B*x+B*y\subseteq V+V\subseteq U$ and hence
$x+y\in U/B\subseteq W$. Taking into account that the topology $\tau^{\sharp\cB}$ is invariant
under inversions, we conclude that $X^{\sharp\cB}$ is a topological group.

Next, we show that the multiplication map $*:R\times X^{\sharp\cB}\to X^{\sharp\cB}$ is separately continuous.

Fix any point $x\in X$ and consider the map $*_x:R\to X^{\sharp\cB}$, $*_x:r\mapsto r* x$. Since
$*_x$ is a homomorphism, it suffices to check the continuity of $*_x$ at zero. Fix any neighborhood
$W\in\tau^{\sharp\cB}$ of zero. By the definition of the topology $\tau^{\sharp\cB}$, there are a
neighborhood $U\subseteq X$ of zero and a bounded subset $B\in\cB$ such that $U/B\subseteq W$. Choose a neighborhood $U_0$ of zero in $X$ such that $U_0+U_0\subseteq U$. The continuity of the map $*_x:R\to X$
yields a neighborhood $V\subseteq R$ of zero such that $V*x\subseteq U_0$. By the boundedness of the set $B$ in $R$, there is a neighborhood $V'\subseteq R$ of zero such
that $BV'\subseteq V$. Then $U_0+B*(V'*x)=U_0+(BV')*x\subseteq U_0+ V*x\subseteq U_0+U_0\subseteq U$ and hence $B*(V'*x)\Subset U$ and finally, $V'*x\subseteq
U/B\subseteq W$.

Next, fix any $r\in R$ and consider the homomorphism ${}^r*:X^{\sharp\cB}\to X^{\sharp\cB}$,
${}^r*:x\mapsto r*x$. The continuity of ${}^r*$ will follow as soon as we check its continuity at
zero. Fix any neighborhood $W\in \tau^{\sharp\cB}$ of zero and find an open neighborhood
$U\subseteq X$ of zero and a bounded set $B\in\cB$ such that $U/B\subseteq W$.

Observe that $U/(Br)$ is a neighborhood of zero in the topology $\tau^{\sharp\cB}$ and $${}^r\!*(U/(Br))=r* \{x\in X:Br*x\Subset U\}\subseteq\{y\in X:B*y\Subset U\}\subseteq U/B\subseteq W.$$
\end{proof}

The following proposition shows that the construction of the $\cB$-modification is functorial in the category of semitopological $R$-modules.

\begin{proposition}\label{p5.4} Let $\mathcal \cB$ be a bornology on a topological ring $R$. Any continuous $R$-linear map $h:X\to Y$ between
semitopological $R$-modules remains continuous as a map $h:X^{\sharp\cB}\to Y^{\sharp\cB}$.
\end{proposition}

\begin{proof} Fix any neighborhood $W\subseteq Y^{\sharp\cB}$ of zero and find an open neighborhood $U\subseteq Y$ of zero and a bounded subset $B\in\cB$ such that $U/B\subseteq W$. Choose a neighborhood $U_0\subseteq Y$ of zero such that $U_0+U_0\subseteq U$. The continuity of the map $h:X\to Y$ yields a neighborhood $V\subseteq X$ of zero such that $h(V)\subseteq U_0$.

We claim that $h(V/B)\subseteq U/B$. Indeed, for any point $x\in V/B$ we get $B*x\subseteq V$ and hence $B* h(x)=h(B*x)\subseteq h(V)\subseteq U_0$. Then $U_0+B* h(x)\subseteq U_0+U_0\subseteq U$, which implies that $B* h(x)\Subset U$ and hence $h(x)\in U/B\subseteq W$.
\end{proof}

Now we show that for a locally $\cB$-bounded topological ring $R$ the bounded modification $X^{\sharp\cB}$ of any semitopological $R$-module $X$ is a topological $R$-module. We recall that a topological ring $R$ is {\em locally $\cB$-bounded} if for any set $B\in\cB$ there is a neighborhood $U\subseteq R$ of zero such that $BU\in\cB$.

\begin{theorem}\label{t5.5} If $R$ is a locally $\cB$-bounded topological ring for some bornology $\cB$ on $R$, then the $\cB$-modification $X^{\sharp\cB}$ of any semitopological $R$-module $X$ is a topological $R$-module.
\end{theorem}

\begin{proof} It suffices to check the continuity of the multiplication $*:R\times X^{\sharp\cB}\to X^{\sharp\cB}$ at zero. Fix a neighborhood $W\in\tau^{\sharp\cB}$ of zero and find a neighborhood $U\subseteq X$ of zero and a bounded set $B\in\cB$ such that $U/B\subseteq W$. Since $R$ is locally $\cB$-bounded, there exists a neighborhood $V\subseteq R$ of zero such that the set $BV$ belongs to the bornology $\cB$. Then the set $U/(BV)$ is a neighborhood of zero in $X^{\sharp\cB}$ such that for any $v\in V$ and $x\in U/(BV)$ we get $B*(v*x)\subseteq (BV)*x\Subset U$, which implies $v*x\in U/B\subseteq W$ and shows that the multiplication map $*:R\times X^{\sharp\cB}\to X^{\sharp\cB}$ is continuous.
\end{proof}

Observe that for any bornologies $\mathcal A\subseteq \cB$ on a topological ring $R$ and any semitopological $R$-module $X$ we get $\tau^{\sharp\!\A}\subseteq\tau^{\sharp\cB}$, which implies that the identity map $X^{\sharp\cB}\to X^{\sharp\!\A}$ is continuous.
In particular, for the bornology $\K$ of precompact subsets of $R$ and the bornology $\M$ of all bounded subsets of $R$ the identity maps
$X^{\sharp\M}\to X^{\sharp\K}\to X$ are continuous. These  two identity maps are homeomorphisms, provided the ring $R$ is crowded and $X$ is a topological $R$-module. We recall that a topological ring $R$ is crowded if any neighborhood of zero in $R$ contains an invertible element of $R$.

\begin{proposition}\label{p5.6} Let $R$ be a crowded topological ring and $\mathcal B$ be a bornology on $R$. For any topological $R$-module $X$, the identity map $\id:X^{\sharp\cB}\to X$ is a homeomorphism.
\end{proposition}

\begin{proof} It suffices to check that for any neighborhood $U\subseteq X$ of zero and any bounded
subset $B\subseteq R$ the set $U/B$ is a neighborhood of zero in $X$. By the continuity of the
addition and multiplication, there exist a neighborhood $U_0\subseteq X$ of zero
in $X$ and a neighborhood $V\subseteq R$ of zero in $R$ such that $U_0+(V* U_0)\subseteq U$. The
boundedness of the set $B$ in $R$ yields a neighborhood $V_0\subseteq V$ such that $B\cdot
V_0\subseteq V$. Since $R$ is crowded, the neighborhood $V_0$ contains an invertible element $r$.
The invertibility of $r$ guarantees that the set $U_r=r* U_0$ is a neighborhood of zero in $X$. We
claim that $U_r\subseteq U/B$. Indeed, for any $x\in U_r$ we can find a point $y\in U_0$ with
$x=r*y$ and conclude that $B*x=B{\cdot}r * y\subseteq B\cdot V_0*U_0\subseteq V*U_0$ and $U_0+B*x\subseteq
U_0+V* U_0\subseteq U$, which implies $x\in U/B$. Therefore, $U/B$ is a neighborhood of zero in the
original topology of $X$ and the identity map $X^{\sharp\cB}\to X$ is a homeomorphism.
\end{proof}

\begin{theorem}\label{t5.7} Let $X$ be a semitopological module over a topological ring $R$.
\begin{enumerate}
\item If $R$ is crowded and locally $\cB$-bounded for some bornology $\cB$ on $R$, then any continuous $R$-homomorphism $h:Z\to X$ defined on a topological $R$-module $Z$ remains continuous as a map to $X^{\sharp\cB}$.
\item If every compact  set in $X$ is $R$-bounded, then any continuous map $f:Z\to X$ defined on a  $k$-space $Z$ remains continuous as a map to $X^{\sharp\K}$, where $\K$ is the bornology of precompact sets in $R$.
\end{enumerate}
\end{theorem}

\begin{proof} 1. Assume that $R$ is crowded and locally $\cB$-bounded for some bornology $\cB$ on $R$. Let $h:Z\to X$ be a continuous $R$-homomorphism defined on a topological $R$-module $Z$. By Proposition~\ref{p5.4}, the map $h:Z^{\sharp\cB}\to X^{\sharp\cB}$ is continuous and by Proposition~\ref{p5.6}, the identity map $\id:Z\to Z^{\sharp\cB}$ is a homeomorphism. Then the map $h:Z\to X^{\sharp\cB}$ is continuous as a composition of two continuous maps.
\smallskip

2. Assume that each compact  set in $X$ in $R$-bounded and let $f:Z\to X$ be a continuous function
defined on a $k$-space $Z$. To show that the map $f:Z\to X^{\sharp\K}$ is continuous it suffices to check that for every
compact set
$K\subseteq Z$ the restriction $f{\restriction}_K:K\to X^{\sharp\K}$ is continuous at each point $z\in K$. Fix
any neighborhood $O_{f(z)}$ of $f(z)$ in $X^{\sharp\K}$  and find an open
neighborhood $U\subseteq X$ of zero and a compact set $B\subseteq R$ such that $f(z)+U/B\subseteq
O_{f(z)}$. Choose a neighborhood $U_0\subseteq X$ of zero such that $U_0+U_0\subseteq U$.

By our assumption, the compact  set $C=f(K)-f(K)$ is $R$-bounded in $X$. By
 Theorem~\ref{t3.2}, the restriction $*{\restriction}_{R\times C}:R\times C\to X$ is continuous. Since $B*\{0\}=\{0\}$ we can use the
compactness of $B$ and find a neighborhood $U_1\subseteq X$ of zero such that $B*(U_1\cap
C)\subseteq U_0$. By the continuity of $f:Z\to X$ at $z$, there is a neighborhood $O_z\subseteq K$ of
the point $z$ such that $f(O_z)\subseteq f(z)+U_1$. We claim that $f(O_z)\subseteq f(z)+U/B$.
Indeed, for any point $y\in O_z$ we get $f(y)-f(z)\in U_1\cap C$ and hence $B*(f(y)-f(z))\in
B*(U_1\cap C)\subseteq U_0$ and $U_0+B* (f(y)-f(z))\subseteq U_0+U_0\subseteq U$,  which implies $f(y)-f(z)\in U/B$ and
$f(O_z)\subseteq f(z)+U/B\subseteq O_{f(z)}$.
\end{proof}

\section{Free semitopological $R$-modules}\label{s6}

In this section we apply the results of the preceding section to study the structure of free
semitopological $R$-modules.

Let $R$ be a topological ring. For a topological space $X$ its {\em free semitopological $R$-module} is a pair $(M_{R}(X),i_X)$ consisting of a semitopological $R$-module $M_R(X)$ and a continuous map $i_X:X\to M_R(X)$ such that for any continuous map $f:X\to Y$ to a semitopological $R$-module $Y$ there exists a unique continuous $R$-homomorphism $\bar f:M_R(X)\to Y$ such that $f=\bar f\circ i_X$.
The standard category arguments show that for each topological space $X$ a free semitopological $R$-module exists and is unique up to an isomorphism.

By analogy we can introduce the notion of a {\em free topological $R$-module} over a topological space.

It turns out that in some cases the free semitopological $R$-module over a topological space coincides with its free topological $R$-module.

\begin{theorem}\label{t6.2} Assume that a topological ring $R$ is locally compact. Then for every $k$-space $X$ its free semitopological $R$-module $M_R(X)$ is a free topological $R$-module.
\end{theorem}

\begin{proof} Let $(M_R(X),i_X)$ be a free semitopological module over $X$ and $M_R(X)^{\sharp\K}$ be
its $\K$-mo\-di\-fi\-ca\-ti\-on with respect to the bornology $\K$ of precompact sets in $R$. Being
locally compact, the topological ring $R$ is locally $\K$-bounded. By Theorem~\ref{t5.5},
$M_R(X)^{\sharp\K}$ is a topological $R$-module. By Theorems~\ref{t3.2}(2) and \ref{t5.7}(2), the
canonical map $i_X:X\to M_R(X)$ remains continuous as a map $i_X:X\to M_R(X)^{\sharp\K}$. The
definition of the free semitopological $R$-module $M_R(X)$ guarantees that the identity map
$M_R(X)\to M_R(X)^{\sharp\K}$ is continuous and hence is a homeomorphism, which implies that
$M_R(X)=M_R(X)^{\sharp\K}$ is a topological $R$-module and a free topological $R$-module over $X$.
\end{proof}

Finally, we present an example of a topological space $X$ whose free semitopological $R$-module is not a topological $R$-module.

\begin{proposition} Let $R$ be a topological ring and $X$ be a semitopological $R$-module which is not a topological $R$-module. Then the free semitopological $R$-module $M_R(X)$ over $X$ is not a topological $R$-module.
\end{proposition}

\begin{proof} The semitopological $R$-module $X$ is a Tychonoff space (being a Hausdorff topological
group), which implies that the canonical map $i_X:X\to M_R(X)$ is a topological embedding. By the
definition of a free semitopological $R$-module, there exists a unique continuous $R$-homomorphism
$h:M_R(X)\to X$ such that $h\circ i_X:X\to X$ is the identity map of $X$. Assuming that $M_R(X)$ is
a topological $R$-module, we conclude that the multiplication map $*:R\times M_R(X)\to M_R(X)$ is
continuous and then the map $\mu:R\times X\to X$, $\mu:(a,x)\mapsto h(a*i_X(X))=a*h(i_X(x))=a* x$
is continuous, too. On the other hand, this map is discontinuous because $X$ is
not a topological $R$-module.
\end{proof}

\section{Acknowledgements}

The authors express their sincere thanks to V.V.~Mykhaylyuk and O.V.~Maslyuchenko for stimulating
discussions concerning co-Namioka spaces.


\newpage

\begin{thebibliography}{}

\bibitem{AT} A.~Arhangel'skii, M.~Tkachenko, {\em Topological groups and related structures}, Atlantis Studies in Mathematics, 1. Atlantis Press, Paris; World Scientific Publishing Co. Pte. Ltd., Hackensack, NJ, 2008.

\bibitem{BMR} T.O.~Banakh, V.K.~Maslyuchenko, A.V.~Ravsky, {\em Semitopological vector spaces}, Math. Bull.
Shevchenko Sci. Soc. {\bf 13} (2016), 84--89.



\bibitem{Bouziad} A.~Bouziad, {\em Notes sur la propri\'et\'e de Namioka}, Trans. Amer. Math. Soc. {\bf 344}:2 (1994), 873--883.

\bibitem{CM} J.~Cao, W.~Moors, {\em Separate and joint continuity of homomorphisms defined on topological groups}, New Zealand J. Math. {\bf 33}:1 (2004), 41--45.

\bibitem{EV} H.R.~Ebrahimi-Vishki, {\em Joint continuity of separately continuous mappings on topological groups}, Proc. Amer. Math. Soc. {\bf 124}:11 (1996), 3515--3518.

\bibitem{En} R.~Engelking, {\em General topology}, Sigma Series in Pure Mathematics, 6. Heldermann Verlag, Berlin, 1989.

\bibitem{Ke} A.S.~Kechris, {\em Classical descriptive set theory}, Springer-Verlag, New York, 1995.

\bibitem{Masl} V.K.Maslyuchenko, {\em Vector spaces with additive tiopology and separately continuous multiplication by scalars}, Sci. Bull. Chernivtsi Univ. {\bf 1}:4 (2011), 95--99 (in Ukrainian).

\bibitem{MP} V.V.~Mykhaylyuk, R.~Pol, {\em On a problem of Talagrand concerning separately continuous functions}, J. Institute Math. Jussieu, (2020); ({\tt doi.org/10.1017/S1474748019000677}).

\bibitem{Rol} S.~Rolewicz,  {\em Metric linear spaces}, PWN, 1984.



\end{thebibliography}
\end{document}